\documentclass[a4paper]{amsart}

\usepackage{amssymb}
\usepackage{amsmath}
\usepackage{latexsym}
\usepackage{amsmath}
\usepackage{euscript}
\usepackage{url}

\begin{document}

\newtheorem{thm}{Theorem}
\newtheorem{lemma}[thm]{Lemma}
\newtheorem{cor}[thm]{Corollary}

\title[Preprint]{A new version of an old modal incompleteness theorem}
\author[Submitted to the Bulletin of the Section of Logic]{Jacob Vosmaer}
\date{February 10, 2006}
\email{contact@jacobvosmaer.nl}
\begin{abstract}
  Thomason \cite{Thomason74} showed that a certain modal logic
  $\mathbf{L}\subset \mathbf{S4}$ is incomplete with respect to Kripke
  semantics. Later Gerson \cite{Gerson75} showed that $\mathbf{L}$ is
  also incomplete with respect to neighborhood semantics. In this
  paper we show that $\mathbf{L}$ is in fact incomplete with respect
  to any class of complete Boolean algebras with operators, i.e.~that
  it is completely incomplete.
\end{abstract}

\maketitle

\section{Introduction}
In 1974, two modal incompleteness theorems were published in the same
issue of the same journal. Fine \cite{Fine74} presented a logic above
$\mathbf{S4}$ and Thomason \cite{Thomason74} presented one between
$\mathbf{T}$ and $\mathbf{S4}$, and both showed that their logics were
incomplete with respect to Kripke semantics. In 1975, a paper by
Gerson \cite{Gerson75} followed in which he showed that both logics
were both also incomplete with respect to neighborhood semantics.
Then, in 2003 Litak \cite{Litak03} showed that Fine's logic is in fact
as he calls it completely incomplete, i.e.~it is incomplete with
respect to any class of Boolean algebras with operators (or BAOs for
short). It is know that Kripke frames correspond to the class of
complete, atomic and completely distributive BAOs and that
neighborhood frames (for normal logics such as the ones we are
considering) correspond to the class of complete, atomic BAOs. In the
present paper, we show what one might almost call a complement to
Litak's result, i.e.~that Thomason's logic is also completely
incomplete. 

\section{An incompleteness theorem}

\subsection{Algebraic preliminaries}
When considering an arbitrary complete BAO $\mathfrak{A}$ below, we
will always assume there is some Kripke frame $\langle W,R\rangle$
such that $\mathfrak{A}=\langle A,\wedge,-,0,\Diamond\rangle$ is a
subalgebra of $\langle
\wp(W),\cap,\phantom{i}^c,\emptyset,m_R\rangle$, where $\phantom{i}^c$
is set-theoretic complementation with respect to $W$ and for
$X\subseteq W$ and $m_R(X):=\{ w\in W\mid \exists v\in X \, (wRv)\}$;
the J\'onsson-Tarski representation theorem tells us that any BAO is
such a subalgebra up to isomorphism. We will make use of a few observations
about suprema in $\mathfrak{A}$.  Let $\{a_n\mid n\in\omega\},\{b_n
\mid n\in\omega\}$ be arbitrary subsets of $\mathfrak{A}$. First of
all, we will use without mentioning the fact that
$\bigcup_{n\in\omega} a_n\leq \bigvee_{n\in\omega} a_n$.  Secondly,
\begin{equation}\label{alg:SumSup}
  \bigcup_{n\in\omega}a_n \subseteq \bigcup_{n\in\omega} b_n \text{ implies
  } \bigvee_{n\in\omega}a_n \leq \bigvee_{n\in\omega} b_n,
\end{equation}
as $\bigcup a_n\leq \bigcup b_n \leq \bigvee b_n$, so $\bigvee a_n$,
being the \emph{least} upperbound of $\{a_n\mid n\in\omega\}$ in
$\mathfrak{A}$, must be below $\bigvee b_n$. Thirdly,
\begin{equation}\label{alg:Disjoint}
\bigcup_{n\in\omega} a_n \cap \bigcup_{n\in\omega} b_n=\emptyset \text{
  implies } \bigvee_{n\in\omega} a_n \wedge \bigvee_{n\in\omega} b_n
=0,
\end{equation}
for if $\bigcup a_n\cap \bigcup b_n=\emptyset$ but $\bigvee a_n \wedge
\bigvee b_n> 0$ then $\bigvee a_n \cap \bigcup b_n>0$. If this were
not the case, then we would get $\bigcup b_n \subseteq \bigvee b_n
\setminus \bigvee a_n\in\mathfrak{A}$, contradicting the fact that
$\bigvee b_n$ is least in $\mathfrak{A}$. So, there must be some $b_i$
such that $b_i\cap \bigvee a_n>0$, and now we know that $\bigcup
a_n\nsubseteq \bigvee a_n \setminus b_i$, for otherwise $\bigvee a_n$
would not be least. It follows that there must be some $a_j$ such that
$a_j \wedge b_i>0$; however this contradicts our assumption that
$\bigcup a_n\cap \bigcup b_n=\emptyset$. It follows that
\eqref{alg:Disjoint} is true.  Finally,
\begin{equation}\label{alg:Distr}
\bigvee_{n\in\omega} \Diamond a_n\leq \Diamond \bigvee_{n\in\omega}
a_n.
\end{equation}
Since for any $k\in\omega$ and $w\in \Diamond A_k=m_R(A_k)$ it must be
the case that $wRv$ for some $v\in A_k \subseteq \bigcup a_n \subseteq
\bigvee a_n$, so that $w\in \Diamond \bigvee a_n$, whence $\bigcup
\Diamond a_n \subseteq \Diamond \bigvee a_n$. It follows that $\bigvee
\Diamond a_n\leq \Diamond \bigvee a_n$.

\subsection{A case of complete incompleteness}
Consider the formulas
\begin{align*}
&A_i:=\Box (q_i\rightarrow r),\\
&B_i:=\Box (r\rightarrow \Diamond q_i)\quad (i=1,2),\\
& C_1:=\Box \neg (q_1\wedge q_2),\\
&A:= r\wedge\Box p\wedge\neg \Box^2
p\wedge A_1\wedge A_2
\wedge B_1\wedge B_2\wedge  C_1\\
&\phantom{A:=}\rightarrow \Diamond(r\wedge\Box(r\rightarrow q_1\vee
q_2),\\
&B:=\Box(p\rightarrow q)\rightarrow (\Box p\rightarrow \Box q),\\
&C:=\Box p\rightarrow p,\\
&D:=(p\wedge \Diamond^2 q)\rightarrow (\Diamond q\vee
\Diamond^2(q\wedge\Diamond p)),\\
&E:=(\Box p\wedge \neg \Box^2 p)\rightarrow \Diamond (\Box^2 p\wedge
\neg\Box^3 p),\\
&F:=\Box p\rightarrow \Box^2 p.
\end{align*}
Let $\mathbf{L}$ be the logic containing all propositional
tautologies, $A,B,C,D$ and $E$ and closed under modus ponens,
substitution and necessitation (this is the same logic as found in
\cite{Thomason74}). It is not hard to see that $\mathbf{T}\subseteq
\mathbf{L}\subseteq \mathbf{S4}$. We will see below that the latter
inclusion is  strict, because $\mathbf{S4}\ni F\notin \mathbf{L}$.
\begin{lemma}
Let $\mathfrak{A}$ be a complete BAO. If $\mathfrak{A}\models L$, then
$\mathfrak{A}\models F$.
\end{lemma}
\begin{proof}
  Let $\mathfrak{A}$ be a complete BAO on which $B$, $C$, $D$ and $E$
  are valid\footnote{We are abusing language here, for we should
    really say $\mathfrak{A}\models B=1$ instead of
    $\mathfrak{A}\models B$. We trust that confusion will not ensue,
    however.}, but $F$ is not. We will show that
  $\mathfrak{A}\not\models A$, proving the statement of the lemma.

  The fact that $\mathfrak{A}\not\models F$ must be witnessed by some
  $a\in \mathfrak{A}$ such that $\Box a\nleq \Box^2 a$. Since by $C$,
  $\Box^2 a\leq \Box a$, it follows that $\Box^2a<\Box a$. For $n\geq
  1$ we define
  \[ b_n:= \Box^n a\setminus \Box^{n+1} a,\] where $c\setminus d:=
  c\wedge -d$. By the above, we already know that $b_1>0$. To
  inductively show that all $b_n>0$, suppose that $b_n>0$, but
  $b_{n+1}=0$. Then substitute\footnote{$\Box^0 a:=a$.} $\Box^{n-1} a$
  for $p$ in $E$, so we get
  \[ b_n= \Box \Box^{n-1} a \wedge - (\Box^2 \Box^{n-1}
  a)\leq \Diamond (\Box^2 \Box^{n-1}a \wedge - (\Box^3\Box^{n-1}
  a))= \Diamond b_{n+1}=\Diamond 0=0, \]
  which is a contradiction, so it must be that $b_{n+1}>0$. This
  completes our induction. Note that if $1\leq i<j$, then since
  $b_j\leq \Box^j a \leq \Box^{i+1} a\leq -b_i$, it must be that $b_i\wedge
  b_j=0$. Next, suppose that 
  \begin{equation}\label{thm:Cover}
    \text{for all } 1\leq i<j\leq  n,\quad   b_i\leq \Diamond b_j
  \end{equation}
  (the base case $n=2$ follows immediately from $E$). We will show
  that \eqref{thm:Cover} must also hold for $n+1$. We only consider
  $j=n+1$ and $i<n$ (for if $i=n$, we can immediately apply $E$ and
  $i,j\leq n$ is already covered by \eqref{thm:Cover}). By our
  induction hypothesis, $b_i\leq \Diamond b_n$ and by $E$, $b_n\leq
  \Diamond b_{n+1}$, so we have $b_i\leq \Diamond^2b_{n+1}$,
  i.e.~$b_i=b_i\wedge \Diamond^2b_{n+1}$. Reverting to definitions, we
  find that $\Diamond b_i=-\Box-(\Box^i
  a\setminus \Box^{i+1} a)$. As $-(\Box^i a\setminus \Box^{i+1}a )\leq
  \Box^{i+1} a$, we get that $\Box-(\Box^i
  a\setminus \Box^{i+1} a)\leq \Box^{i+2} a$, so $\Diamond b_i\wedge
  \Box^{i+2} a=0$. Since also $b_{n+1}\leq \Box^{n+1} a\leq
  \Box^{i+2}a$ (as $i<n$), it follows that $b_{n+1}\wedge \Diamond
  b_i=0$, so substituting $b_i$ for $p$ and $b_{n+1}$ for $q$ in $D$,
  we find that
  \[ b_i=b_i\wedge \Diamond^2 b_{n+1}\leq \Diamond b_{n+1}\vee
  \Diamond^2(b_{n+1}\wedge \Diamond b_i)=\Diamond b_{n+1}\vee
  \Diamond^2 0=\Diamond b_{n+1}.\] It follows that \eqref{thm:Cover}
  holds for $n+1$, so by induction \eqref{thm:Cover} is true for all
  $n\geq 2$.

  Now we define the following elements of $\mathfrak{A}$: 
  \[ p:=a, \quad q_i:=\bigvee_{n\geq 0} b_{3n+i} \quad(i=1,2,3),\quad
  r:=\bigvee_{n\geq 1} b_n.\] (Note that this is where we use the
  assumption that $\mathfrak{A}$ is complete.) We will use these
  elements to show that $A$ is not valid. First of all, as
  $\bigcup_{n\geq 0} b_{3n+i}\subseteq \bigcup_{n\geq 1} b_n$, it
  follows by \eqref{alg:SumSup} that $q_i\leq r$ for $i=1,2,3$, so
  $q_i\rightarrow r=1$, whence $A_1=A_2=\Box 1=1$. Secondly, by
  \eqref{thm:Cover}, for any $n\geq 1$ there must exist a $k\in\omega$
  such that $b_n\leq \Diamond b_{3k+i}$, whence $\bigcup_{n\geq 1} b_n
  \subseteq \bigcup_{n\geq 0} \Diamond b_{3n+i}$. By
  \eqref{alg:SumSup}, this means that
  \[ r= \bigvee_{n\geq 1} b_n \leq \bigvee_{n\geq 0} \Diamond
  b_{3n+i}\leq \Diamond\bigvee_{n\geq 0} b_{3n+i}=\Diamond q_i,\]
  where the latter inequality follows from \eqref{alg:Distr}.
  Therefore, $r\rightarrow \Diamond q_i=1$, so $B_1=B_2=\Box 1=1$.
  Finally, as $\bigcup_{n\geq 0} q_{3n+i} \cap \bigcup_{n\geq 0}
  q_{3n+j}=\emptyset$ if $1\leq i< j\leq 3$, it follows by
  \eqref{alg:Disjoint} that $q_i\wedge q_j=0$ for $1\leq i< j\leq 3$,
  whence $C_1=\Box-0=1$. Combining all this, we find that
  \[r\wedge \Box p\wedge -\Box^2 p\wedge A_1\wedge A_2 \wedge B_1 \wedge
  B_2\wedge C_1=r\wedge (\Box a\setminus \Box^2a)=b_1.\]
  However, we have $r=q_1\vee q_2\vee q_3$, and as the $q_i$ are
  disjoint, this means that $r\wedge -q_1\wedge -q_2=q_3$. By the above,
  $r\leq \Diamond q_3$, so
  \[ 0=r\wedge -\Diamond q_3=r\wedge\Box -q_3= r\wedge\Box-(r\wedge
  -q_1\wedge -q)=r\wedge\Box(r\rightarrow q_1\vee q_2).\] It follows
  that $\Diamond (r\wedge \Box(r\rightarrow q_1\vee q_2))=0$,
  contradicting $A$ as $b_1>0$. We conclude that
  $\mathfrak{A}\not\models A$.
\end{proof}
For $\mathcal{C}$ some class of BAOs, we define $\Delta
\models_{\mathcal{C}} \Gamma$ if for every $\mathfrak{A}\in
\mathcal{C}$, $\mathfrak{A}\models \Delta$ only if
$\mathfrak{A}\models \Gamma$.
\begin{cor}
  Let $\mathcal{C}$ be any class of complete BAOs. Then
  $\{A,B,C,D,E\}\models_{\mathcal{C}} F$.
\end{cor}
\begin{lemma}
  $F\notin \mathbf{L}$.
\end{lemma}
\begin{proof} 
  See \cite{Thomason74}.  Thomason proofs the lemma by showing that
  the veiled recession frame, which is in fact (as it should be) an
  incomplete BAO, validates $\mathbf{L}$ while $\neg F$ can be
  satisfied on it.
\end{proof} 
The lemmas give us the following:
\begin{thm}
  $\mathbf{L}$ is completely incomplete.
\end{thm}

\section{Acknowledgements}
This paper is the result of Eric Pacuit asking me to write some paper
about neighborhood semantics for modal logic for a class of his. I am
the first to admit that the connection between the present result and
neighborhood semantics is lateral at best, but my blatant disregard
for the assigned subject matter notwithstanding he helped me out
gladly on several occasions. For this I thank him. I thank my fellow
Master of Logic students Gaelle Fontaine and Christian Kissig for
discussions.

\end{document}